\documentclass[11pt,a4paper]{article}

\usepackage{epsf,epsfig,amsfonts,amsgen,amsmath,amstext,amsbsy,amsopn,amsthm,cases,listings
}
\usepackage{ebezier,eepic}
\usepackage{color}
\usepackage{multirow}
\usepackage{epstopdf}
\usepackage{graphicx}
\usepackage{pgf,tikz}
\usepackage{mathrsfs}
\usepackage{authblk}
\usepackage{lineno}
\usepackage{enumitem}
\usetikzlibrary{arrows}

\newtheorem{dfn}{Definition} [section]
\newtheorem{theorem}[dfn]{Theorem}
\newtheorem{lemma}[dfn]{Lemma}

\newenvironment{pf}{\noindent{\bf Proof}}

\def\qed{\hfill \rule{4pt}{7pt}}

\def\lc{\left\lceil}

\def\rc{\right\rceil}

\def\lf{\left\lfloor}

\def\rf{\right\rfloor}

\begin{document}

\title{\bf\Large New Short Proofs to Some Stability Theorems}
\date{\today}
\author{Xizhi Liu \thanks{Department of Mathematics, Statistics, and Computer Science, University of Illinois, Chicago, IL, 60607 USA.\ Email: \textbf{xliu246@uic.edu}}}
\maketitle
\begin{abstract}
We present new short proofs to both the exact and the stability results of two extremal problems.
The first one is the extension of Tur\'{a}n's theorem in hypergraphs, which was firstly studied by Mubayi $\cite{MU06}$.
The second one is about the cancellative hypergraphs, which was firstly studied by Bollob\'{a}s $\cite{BO74}$
and later by Keevash and Mubayi $\cite{KM04}$.
Our proofs are concise and straightforward, but give a sharper version of stability theorems to both problems.
\end{abstract}

\section{Introduction}
Let $H$ be an $n$-vertex $r$-graph and let $\mathcal{F}$ be a family of $r$-graphs.
$H$ is \textit{$\mathcal{F}$-free} if it does not contain any $r$-graph in $\mathcal{F}$ as a subgraph.
The \textit{Tur\'{a}n number} $ex(n,\mathcal{F})$ is the maximum number of edges in an $n$-vertex $\mathcal{F}$-free $r$-graph.
$\mathcal{F}$ is called \textit{non-degenerate} if the \textit{Tur\'{a}n density} $\pi(\mathcal{F}):= \lim_{n\to \infty} ex(n,\mathcal{F})/\binom{n}{r}$ is not $0$.

Determining, even asymptotically, the value of $ex(n,\mathcal{F})$ for general non-degenerate $r$-graphs $\mathcal{F}$ with $r\ge 3$ is known to be notoriously hard.
On the other hand, many families $\mathcal{F}$ have the property that there is a unique extremal family attain the value $ex(n,\mathcal{F})$,
and any $\mathcal{F}$-free hypergraph with close to $ex(n,\mathcal{F})$ edges is also structurally close to the extremal family.
This property of $\mathcal{F}$ is called \textit{stability}.
It is both an intersecting property of $\mathcal{F}$ and also an extremely useful tool in determining the value of $ex(n,\mathcal{F})$.
The Tur\'{a}n numbers for many families $\mathcal{F}$ has been determined by using this method, and we refer the reader to a survey by Keevash $\cite{KE11}$ for results before 2011.

In the present paper, we mainly focus on the stability properties for two extremal problems.
The first one is the extension of Tur\'{a}n's theorem in hypergraphs, and it was firstly studied by Mubayi $\cite{MU06}$.

Let $V_1\cup ...\cup V_{\ell}$ be a partition of $[n]$ with each part of size either $\lf n/{\ell} \rf$ or $\lc n/{\ell} \rc$.
$T_{r}(n,{\ell})$ is the family of all $r$-sets that intersect each $V_i$ in at most one vertex.
Let $t_{r}(n,{\ell})$ denote the number of edges in $T_{r}(n,{\ell})$.
$\mathcal{K}_{{\ell}+1}^{(r)}$ is the family of all $r$-graphs $F$ with at most $\binom{{\ell}+1}{2}$
edges such that for some $({\ell}+1)$-set $S$ every pair $x,y\in S$ is covered by an edge of $F$.
Notice that $T_{2}(n,{\ell})$ is just the ordinary Tur\'{a}n graph, and $\mathcal{K}_{{\ell}+1}^{(2)}$
is just the ordinary complete graph on ${\ell}+1$ vertices, which is also denoted by $K_{{\ell}+1}$.

In $\cite{MU06}$ Mubayi proved both the exact and stability result for $\mathcal{K}_{{\ell}+1}^{(r)}$-free $r$-graphs.
\begin{theorem}[Mubayi, \cite{MU06}]
Let $n, {\ell}, r\ge 2$. Then
\[
ex(n,\mathcal{K}_{{\ell}+1}^{(r)})=t_{r}(n,{\ell})
\]
and $T_{r}(n,{\ell})$ is the unique maximum $\mathcal{K}_{{\ell}+1}^{(r)}$-free $r$-graph on $n$ vertices.
\end{theorem}

\begin{theorem}[Stability; Mubayi, \cite{MU06}]
Fix $l\ge r \ge 2$.
For every $\delta >0$,
there exists an $\epsilon>0$ and an $n_0$ such that the following holds for all $n\ge n_0$.
Let $G$ be an $n$-vertex $\mathcal{K}_{{\ell}+1}^{(r)}$-free $r$-graph
with at least $(1-\epsilon)t_{r}(n,{\ell})$ edges.
Then the vertex set of $G$ has a partition $V_1\cup \ldots \cup V_{\ell}$ such that
all but at most $\delta n^r$ edges have at most one vertex in each $V_i$.
\end{theorem}

Note that in $\cite{MU06}$ Mubayi did not give an explicit relation between $\epsilon$ and $\delta$,
but our proof will show that it suffices to choose $\epsilon = (r-2)! \delta$.
Also, note that in $\cite{DE17}$ Contiero, Hoppen, et al. also proved a linear dependence between $\delta$ and $\epsilon$ by induction on ${\ell}+r$,
but our proof is different and much shorter.

\bigskip

The second one is about the cancellative hypergraphs, and it was firstly studied by Bollob\'{a}s $\cite{BO74}$
and later by Keevash and Mubayi $\cite{KM04}$.

A hypergraph $H$ is called cancellative if it does not contain three distinct sets $A,B,C$ with $A\triangle B\subset C$.
Note that an ordinary graph $G$ is cancellative iff it does not contain a triangle (i.e. $K_3$),
and Mantel's theorem states that the maximum size of a cancellative graph is uniquely achieved by $T_{2}(n,2)$.
Motivated by Mantel's theorem, in the 1960's, Katona raised the question of determining the
maximum size of a cancellative $3$-graph and conjectured that the maximum size of
a cancellative $3$-graph is achieved by $T_{3}(n,3)$.
Katona's conjectured was proved by Bollob\'{a}s in $\cite{BO74}$.

\begin{theorem}[Bollob\'{a}s, \cite{BO74}]
A cancellative $3$-graph on $n$ vertices has at most $t_{3}(n,3)$ edges, with equality only for $T_{3}(n,3)$.
\end{theorem}

In $\cite{KM04}$ a new proof of Bollob\'{a}s' result was given by Keevash and Mubayi,
and they also proved a stability theorem for cancellative $3$-graphs.

\begin{theorem}[Stability; Keevash and Mubayi, \cite{KM04}]
For any $\delta>0$ there exists $\epsilon>0$ and $n_0$ such that
the following holds for all $n\ge n_0$.
Any $n$-vertex cancellative $3$-graph with  at least $(1-\epsilon)t_3(n,3)$ edges
has a partition of  vertex set as $[n]=V_1\cup V_2\cup V_3$ such that
all but at most $\delta n^3$ edges of $H$ has one vertex in each $V_i$.
\end{theorem}

In their proof they also gave an explicit relation between $\epsilon$ and $\delta$, which is $\epsilon < 27/2 \times 10 ^{-24} \delta^6$.
Our proof will show that it suffices to choose $\epsilon =\delta/100$.

\medskip

The rest of this paper is organized as following.
In Section 2 we introduce some definitions, useful theorems and lemmas.
In Section 3 we prove Theorems 1.1 and 1.2.
In Section 4 we prove Theorems 1.3 and 1.4.
In Section 5 we present a short proof to the stability of a generalized Tur\'{a}n problem in graph theory.
In the last section we present a brief discussion about the relation between $\epsilon$ and $\delta$.

\section{Preliminaries}
Let $H$ be an $r$-graph on $[n]$.
The \textit{size} of $H$ is the number of edges in $H$, which is denoted by $|H|$.
$I\subset [n]$ is an \textit{independent set} if every edge  in $H$ contains at most one vertex of $I$.
The \textit{shadow} of $H$, denoted by $\partial H$, is defined as
\[
\partial H=\left\{A  \in \binom{[n]}{r-1}: \exists B\in H\text{ such that } A\subset B \right\}
\]
For every nonempty set $S\subset [n]$, define the \textit{link} $L(S)$ of $S$ in $H$ to be
\[
L(S)=\left\{A\in\partial H: A\cup\{s\}\in H,\ \forall s\in S \right\}
\]
For convenience, we use $L(u)$ to represent $L(\{u\})$, and use $L(u,v)$ to represent $L(\{u,v\})$.
Note that in our proof $L(u,u)$ also represents $L(\{u\})$.

Let $T\in \partial H$, the  \textit{neighborhood} of $T$ in $H$ is defined as
\[
N\left(T\right)=\left\{v\in[n]: T\cup \{v\}\in H \right\}
\]
and the \textit{degree} of $T$ is $d\left(T\right)=|N\left(T\right)|$.
It follows from an easy double counting that
\[
\sum_{T\in \partial H}d\left(T\right)=3|H|
\]

The edge set of an ordinary graph $G$ can be viewed as a family of unordered pairs.
To keep the calculations in our proof simply,
we define an auxiliary family $\vec{G}$ of order pairs as $\vec{G}=\{(u,v): \{u,v\}\in G\}$.
Note that if $\{u,v\}\in G$, then $(u,v)$ and $(v,u)$ are both contained in $\vec{G}$
and hence we have $|\vec{G}|=2|G|$.
Let $N$ be a set, we use $N^2$ to denote the cartesian product $N\times N$,
which is also the collection of all ordered pairs $(u,v)$ with $u,v\in N$.
Here $u$ and $v$ might be the same.

Our proof of theorems 1.1 and 1.2 is based on two results.
The first one is the stability of $K_{{\ell}+1}$-free graphs.

\begin{theorem}[F\"{u}redi, \cite{FU15}]
Let $t\ge 0$ and let $G$ be an $n$-vertex $K_{{\ell}+1}$-free graph with $t_{2}(n,{\ell})-t$ edges.
Then $G$ contains an ${\ell}$-partite subgraph $G'$ with at least $t_{2}(n,{\ell})-2t$ edges.
\end{theorem}

The second one describes an relation between the number of copies of $K_{r_1}$ and $K_{r_2}$ in a $K_{{\ell}+1}$-free graph,
where $r_1$ and $r_2$ are two positive integers less that ${\ell}+1$.

\begin{theorem}[Fisher and Ryan, \cite{FR92}]
Let $G$ be an $n$-vertex $K_{{\ell}+1}$-free graph.
For every $i\in [{\ell}]$,
let $k_i$ denote the number of copies of $K_{i}$ in $G$.
Then
\[
\left( \frac{k_{\ell}}{\binom{{\ell}}{{\ell}}}\right)^{\frac{1}{{\ell}}}\le \left( \frac{k_{{\ell}-1}}{\binom{{\ell}}{{\ell}-1}}\right)^{\frac{1}{{\ell}-1}}\le ... \le \left( \frac{k_2}{\binom{{\ell}}{2}}\right)^{\frac{1}{2}}\le \left( \frac{k_1}{\binom{{\ell}}{1}}\right)^{\frac{1}{1}} \eqno{(1)}
\]
\end{theorem}

\bigskip

To prove theorems 1.3 and 1.4 we first present two simply properties of cancellative $3$-graphs.

\begin{lemma}
Let $H$ be a cancellative $3$-graph, and $v$ is a vertex in $H$.
Then the link graph $L(v)$ is triangle-free.
\end{lemma}
\begin{proof}
Suppose $\{x,y,z\}$ is a triangle in $L(v)$.
Then $\{v,x,y\},\{v,x,z\},\{v,y,z\}$ are all contained in $H$, but
\[
\{v,x,y\}\triangle\{v,x,z\}=\{y,z\}\subset \{v,y,z\}
\]
which is a contradiction. Therefore, $L(v)$ is triangle-free.
\end{proof}

\begin{lemma}
Let $H$ be a cancellative $3$-graph, and $T\in \partial H$.
Then $N(T)$ is an independent set.
\end{lemma}
\begin{proof}
Let $u,v\in N(T)$ and let $A_1=\{u\}\cup T$ and $A_2=\{v\}\cup T$.
Note that $A_1$ and $A_2$ are contained in $H$.
Since $A_1\triangle A_2=\{u,v\}$ and by assumption there is no edge in $H$ containing $\{u,v\}$.
Therefore, $N(T)$ is an independent set.
\end{proof}

In the proof of theorem 1.4 we need the following lemma,
which is essentially the stability of triangle-free graphs.
For completeness we include its proof here.

Let $G$ be an ordinary graph and let $v$ be a vertex in $G$.
We use $N_{G}(v)$ to denote the neighborhood of $v$ in $G$, and use $d_{G}(v)$ to
denote the degree of $v$ in $G$.
\begin{lemma}
Let $G$ be a triangle-free graph on $[n]$ with at least $(1-\epsilon)(n/2)^2$ edges.
Then $G$ contains two vertices $v_1$ and $v_2$ such that $N_{G}(v_1)$ and $N_{G}(v_2)$ are disjoint
and $|N_{G}(v_1)|+ |N_{G}(v_2)|\ge (1-\epsilon)n$.
\end{lemma}
\begin{proof}
Since $G$ is triangle-free.
So $N_{G}(u)$ and $N_{G}(v)$ are disjoint for all edge $uv$ in $G$.
Therefore, it suffices to find an edge $uv$ in $G$ such that $d_{G}(u)+ d_{G}(v) \ge (1-\epsilon)n$.
Combining an easy counting argument with the Jensen Inequality we obtain
\[
\sum_{uv\in E(G)}\left(d_{G}(u)+ d_{G}(v) \right) = \sum_{v\in V(G)} d^2_{G}(v) \ge \frac{\left( \sum_{v\in V(G)} d_{G}(v) \right)^2}{n} = \frac{4 e^2(G)}{n}
\]
It follows from an averaging argument that there exists an edge $uv$ with $d_{G}(u)+ d_{G}(v)\ge 4 e(G)/n \ge (1-\epsilon)n$.
\end{proof}


\section{Proofs of Theorems 1.1 and 1.2}
Let $H$ be a $\mathcal{K}_{{\ell}+1}^{(r)}$-free $r$-graph on $[n]$.
Define an auxiliary graph
\[
G=\left\{A\in\binom{[n]}{2}: \exists B\in H\text{ such that }A\in B  \right\}
\]
Let us state two easy facts about the relation between $H$ and $G$ without proof.
\begin{lemma}
\begin{enumerate}[label=(\alph*).]
\item $H$ is $\mathcal{K}_{{\ell}+1}^{(r)}$-free iff  $G$ is $K_{{\ell}+1}$-free.
\item The number of edges in $H$ is at most the number of copies of $K_r$ in $G$.
\end{enumerate}
\end{lemma}
\qed

\noindent\textbf{Proof of theorem 1.1:}
Combining lemma $3.1$ with equation $(1)$, we obtain that $|H|\le \binom{{\ell}}{r}\left( \frac{n}{{\ell}}\right)^{r}$.
This proves theorem 1.1 for the case ${\ell}\text{ divides }n$, and we omit the proof of the other case.
\qed

\medskip

\noindent\textbf{Proof of theorem 1.2:} Choose $\epsilon=(r-2)!\delta$, and let $n$ be sufficiently large.
By assumption we have $|H|\ge (1-\epsilon)t_{r}(n,{\ell})\ge (1-2\epsilon)\binom{{\ell}}{r}\left( n/{\ell}\right)^{r}$.
Combining lemma $3.1$ with equation $(1)$ we know that the number of edges $e$  in $G$ satisfies
\[
e\ge (1-2\epsilon)^{2/r}\binom{{\ell}}{2}\left(\frac{n}{{\ell}}\right)^2\ge (1-2\epsilon)\binom{{\ell}}{2}\left( \frac{n}{{\ell}}\right)^2\ge (1-2\epsilon)t_{2}(n,{\ell})
\]
Therefore, by theorem 2.1, $G$ has a vertex set partition $V_1\cup ...\cup V_{\ell}$ such that
all but at most $2\epsilon t_{2}(n,{\ell})$ edges of $G$ have at most one vertex in each $V_i$.
It follows that all but at most $2\epsilon t_{2}(n,{\ell})\binom{n}{r-2}\le \epsilon n^r/(r-2)!\le \delta n^r$ edges of $H$ have at most one vertex in each $V_i$. This completes the proof of theorem 1.2.
\qed


\section{Proofs of Theorems 1.3 and 1.4}
The most improtant step in this section is building an relation between $H$ and $\partial H$, which is equation $(2)$.

\noindent\textbf{Proof of theorem 1.3:}
Let us count the number of ordered pairs  $(u,v)$ in  $[n]^2\setminus\overrightarrow{\partial H}$.
By lemma 2.4, if $\{u,v\}$ is contained in $N(e)$ for some $e\in \partial H$, then $\{u,v\}$ can not be contained in ${\partial H}$.
Since every set $S\subset [n]$ is contained in exactly $|L(S)|$ sets in $\{ N(T):  T\in \partial H\}$.
Therefore, we have
\[
\sum_{T\in \partial H}\sum_{(u,v)\in N^2(T)}\frac{1}{|L(u,v)|}\le n^2-2|\partial H| \eqno{(2)}
\]
Combining lemma 2.3 with Mantel's theorem we obtain that  $|L(u,v)|\le \left(n-d(T)\right)^2/4$ for every $(u,v)\in[n]^2$.
It follows from $(2)$ that
\[
\sum_{T\in \partial H}4\left(\frac{d(T)}{n-d(T)}\right)^2 \le n^2- 2|\partial H|
\]
Since $\left(x/(n-x)\right)^2$ is convex for $x\in [0,n]$, it follows from Jensen's inequality that
\[
4\left(\frac{ 3|H|/|\partial H| }{ n-3|H|/|\partial H| } \right)^2 |\partial H|\le n^2-2| \partial H|
\]
Now let $z=\frac{ 3|H|/|\partial H| }{ n-3|H|/|\partial H| }$.
Then $(4)$ implies
\[
|\partial H|\le \frac{n^2}{2(2z^2+1)}
\]
Substitute $|H|=\frac{zn}{3(z+1)}|\partial H|$ into the equation above  we obtain
\[
|H|\le \frac{z}{6(z+1)(2z^2+1)} n^3
\]
Since the maximum of $\frac{z}{6(z+1)(2z^2+1)}$ is $1/27$.
Therefore, we have $|H|\le \left(\frac{n}{3}\right)^3$.
This proves theorem 1.3 for the case $3$ divides $n$, and we omit the proof of the other case.
\qed

\bigskip

Choose $\epsilon = \delta/100$.
Let $H$ be a cancellative $3$-graph on $[n]$ with at least $(1-\epsilon) t_3(n,3)> (1-2\epsilon)(n/3)^3$ edges.
Before we prove theorem 1.4, let us present a lemma follows from equation $(2)$.

\begin{lemma}
There exists $T\in \partial H$ such that
\[
\sum_{(u,v)\in N^2(T)}|L(u,v)|\ge (1- 100 \epsilon)d^2(T)\left(\frac{n-d(T)}{2}\right)^2 \eqno{(3)}
\]
\end{lemma}
\begin{proof}
Suppose that $(3)$ is false for all $T\in \partial H$.
Since $1/x$ is convex for $x>0$, it follows from Jensen's inequality that
\[
\sum_{(u,v)\in N^2(T)}\frac{1}{|L(u,v)|}\ge \frac{ d^2(T)}{\sum_{(u,v)\in N^2(T)} |L(u,v)|/d^2(T)}>\frac{ 4d^2(T)}{(1-100 \epsilon)\left(n-d(T)\right)^2} \eqno{(4)}
\]
Substitute $(4)$ into $(2)$ we obtain
\[
\sum_{T\in \partial H}\frac{ 4 d^2(T)}{(1-100 \epsilon)\left(n-d(T)\right)^2}\le n^2-2|\partial H|
\]
Similar argument as in the proof of theorem 1.3 yields
\[
|H|\le \frac{z}{6(z+1)\left(\frac{2z^2}{1-100 \epsilon}+1\right)} n^3
\]
By assumption we have $\frac{3 |H| }{|\partial H|} \ge \frac{3 (1- \epsilon )(n/3)^3}{n^2/2}\ge 1/9$.
Therefore, we may assume that $z>1/8$.
It follows that $\frac{2z^2}{1- 100 \epsilon }+1 > \frac{2z^2+1}{1- 2 \epsilon}$.
So we obtain
\[
\frac{z}{6(z+1)\left(\frac{2z^2}{1- 100 \epsilon }+1\right)}< (1- 2 \epsilon)\frac{z}{6(z+1)\left(2z^2+1\right)}\le \frac{1}{27}(1-2\epsilon )
\]
This implies that  $|H|< (1- 2 \epsilon)\left(\frac{n}{3} \right)^3<(1-\epsilon)t_3(n,3)$, which is a contradiction.
\end{proof}

\bigskip

\noindent\textbf{Proof of theorem 1.4:}
Choose $T\in \partial H$ such that $(3)$ holds for $T$.
Let $V_1'= N(T)$.
By Pigeonhole principle, there exists a pair $(u,v)\in N^2(T)$ such that $|L(u,v)|\ge (1- 100 \epsilon )\left((n-d(T))/{2}\right)^2$.
Let $L$ denote the graph $L(u,v)$ and let $U$ denote the vertex set $[n]\setminus N(T)$.
Combining lemma 2.4 with lemma 2.5 we know that there exist two vertices $x$ and $y$ in $U$
such that $N_{L}(x)$ and $N_{L}(y)$ are disjoint and $N_{L}(x)+ N_{L}(y)\ge (1- 100 \epsilon )(n-d(T))$.
Let $V_2= N_{L}(x)$ and $V_3= N_{L}(y)$.
Note that $N_{L}(x)= N(ux)$ and $N_{L}(y)= N(uy)$ and hence $V_2$ and $V_3$ are independent sets in $H$.

Now we have independent sets $V_1',V_2$ and $V_3$, and $|V_1'|+|V_2|+|V_3|\ge d(T) + (1- 100 \epsilon )(n-d(T)) > n- 100 \epsilon n$.
Let $V_1= [n]\setminus (V_2\cup V_3)$.
The number of edges in $H$ that has at least two vertices in some $V_i$
is at most $\binom{100 \epsilon n}{3}+\binom{100 \epsilon n}{2}\binom{n}{1}+\binom{100 \epsilon n}{1}\binom{n}{2}< 100 \epsilon n^3= \delta n^3$.
This completes the proof of theorem 1.4.
\qed


\section{Further Applications}
In this section we present some applications of equation $(1)$ in the generalized Tur\'{a}n problems.

Let $T$ and $H$ be two ordinary graphs.
Let $ex(n, T, H)$ denote the maximum possible number of copies of $T$
in an ordinary $H$-free graph on $n$ vertices.
The function $ex(n,T,H)$ is called the generalized Tur\'{a}n number.

Fix ${\ell}\ge r \ge 3$.
In $\cite{ER62}$ Erd\H{o}s proved that $ex(n, K_r, K_{{\ell}+1})\le t_{r}(n,{\ell})$.
Actually a similar argument as in the proofs of theorems 1.1 and 1.2 also gives
an exact and stability result to $ex(n,K_r,K_{{\ell}+1})$.
Here we state the stability result without proof.

\begin{theorem}
Fix ${\ell}\ge r\ge 3$, and $\delta>0$.
Then there exists an $\epsilon>0$ and an $n_0$ such that the following holds for all $n\ge n_0$.
If $G$ is an $n$-vertex $K_{{\ell}+1}$-free graph containing at least $(1-\epsilon)\binom{{\ell}}{r}t_r(n,{\ell})$ copies of $K_r$,
then $G$ has a vertex set partition $V_1\cup \ldots \cup V_{\ell}$ such that all but at most $\delta n^2$ edges have at most one vertex in
each $V_i$.
\end{theorem}
\qed

Note that our proof implies that it suffices to choose $\epsilon = \delta$.

\bigskip

In $\cite{AS16}$ Alon and Shikhelman studied the function $ex(n, T, H)$ for other combinations of $T$ and $H$.
In particular they proved that $ex(n,K_r, H)= (1+o(1))t_{r}(n,{\ell})$ holds for every graph $H$ with chromatic number $\chi(H)= {\ell}+1$.
Later their result was improved by Ma and Qiu $\cite{MQ18}$, who proved that $ex(n,K_r, H)= t_{r}(n,{\ell})+\text{biex}(n,H)\cdot \Theta(n^{r-2})$,
where $\text{biex}(n,H)$ is the Tur\'{a}n number of the decomposition family of $H$.
Moreover they proved a stability result for $ex(n,K_r, H)$.
\begin{theorem}[Ma and Qiu, \cite{MQ18}]
Fix ${\ell}\ge r\ge 3$, and $\delta>0$.
For every graph $H$ with chromatic number ${\ell}+1$,
there exists an $\epsilon>0$ and an $n_0$ such that the following holds for all $n\ge n_0$.
If $G$ is an $n$-vertex $H$-free graph containing at least $(1-\epsilon)\binom{{\ell}}{r}t_r(n,{\ell})$ copies of $K_r$,
then $G$ has a vertex set partition $V_1\cup \ldots \cup V_{\ell}$ such that all but at most $\delta n^2$ edges have at most one vertex in
each $V_i$.
\end{theorem}
Here we present a short proof to theorem 5.2 using theorem 5.1 and the Removal Lemma, and our proof implies that it is suffices to choose $\epsilon= \delta/3$.
\begin{theorem}[Removal Lemma, e.g. see \cite{FU15}, \cite{FO11}]
Let $H$ be a graph with chromatic number ${\ell}+1$.
For every $\delta >0$ there exists an $n_0$ such that the following holds for all $n\ge n_0$.
Every $n$-vertex $H$-free graph $G$ can be made $K_{{\ell}+1}$-free by removing
at most $\delta n^2$ edges.
\end{theorem}

\noindent\textbf{Proof of Theorem 5.2: }
Let $n$ be sufficiently large. Choose $\epsilon = \delta/3$.
Let $G$ be an $n$-vertex $H$-free graph containing at least $(1-\epsilon)\binom{{\ell}}{r}t_r(n,{\ell})$ copies of $K_r$.
By the Removal Lemma, $G$ contains a $K_{{\ell}+1}$-free subgraph $G'$ with at least $e(G)-\epsilon n^2/{\ell}^r$ edges.
Since every edge $e$ in $G$ is contained in at most $\binom{n}{r-2}$ copies of $K_{r}$ in $G$.
Therefore, the number of copies of $K_r$ in $G'$ is at least $(1-2\epsilon)\binom{{\ell}}{r}t_r(n,{\ell})$.
By theorem 5.1, $G'$ has a vertex partition $V_1\cup \ldots \cup V_{\ell}$ such that all but at most
$2\epsilon n^2$ edges in $G'$ have at most one vertex in each $V_i$.
Therefore, all but at most $3 \epsilon n^2$ edges in $G$ have at most one vertex in each $V_i$.
\qed


\section{Concluding Remarks}
Note that we showed that a linear dependence between $\delta$ and $\epsilon$ is sufficient for Theorems 1.2, 1.4, 5.1 and 5.2,
and in $\cite{FU15}$ F\"{u}redi showed that a linear dependence is also sufficient for Theorem 2.1.
So one might wondering if the linear dependence between $\delta$ and $\epsilon$ is tight (up to a constant) for the stability theorems above.
In other words, if there exists an absolute constant $C>0$ such that for every $\epsilon>0$ there exists a construction with $\delta \ge C \epsilon$.

We did not try to answer the question above in full generality, but our example  below  of $K_3$-free graphs shows that the answer seems to be negative.

\bigskip

Fix $\epsilon>0$.
Let $G=(V,E)$ be an $n$-vertex $K_3$-free graph with $\left(1/4-\epsilon \right)n^2$ edges.
Let $V_1\cup V_2$ be a partition of $V$ such that the number of edges in the bipartite graph $G[V_1,V_2]$ is maximum.
Define the set of \textit{bad edges} $B$ and the set of \textit{missing edges} $M$ as following.
\[
B=\{uv\in E(G): u,v\in V_i \text{ for some } i\in \{1,2\}\}
\]
and
\[
M=\{uv\not\in E(G): u\in V_1\text{ and } v\in V_2\}
\]
Therefore, in order to make $G$ bipartite one has to remove all edges in $B$.

Assume that $|B|= \delta n^2$.
Let $B_1= B\cap \binom{V_1}{2}$ be the set of bad edges contained in $V_1$.
Without lose of generality we may assume that $|B_1|\ge \delta n^2/2$.

For every vertex $v\in V_1$, let $N_1(v)$ be the neighborhood of $v$ in $V_1$, and let $d_1(v)= |N_1(v)|$.
Let $N_2(v)$ be the neighborhood of $v$ in $V_2$, and let $d_2(v)= |N_2(v)|$.
By the maximality of the partition $V_1\cup V_2$, we know that $d_2(v)\ge d_1(v)$ since otherwise one can move $v$ from $V_1$ to $V_2$ to get a larger bipartite subgraph of $G$.
Also we know that there is no edge between $N_{1}(v)$ and $N_2(v)$ since $G$ is $K_3$-free.

\medskip

Now let $\Delta = \max\{d_1(v): {v\in V_1}\}$.

\medskip

\textbf{Case 1: } $\Delta \ge \delta^{1/3}n$.
Then choose $v\in V_1$ of maximum degree $\Delta$.
Since there is no edge between $N_{1}(v)$ and $N_2(v)$.
Therefore, $|M|\ge (\Delta n)^2\ge \delta^{2/3} n^2$.
On the other hand, we have $|M|\le \epsilon n^2 + \delta n^2$.
So
\[
\delta^{2/3} \le \epsilon + \delta
\]
which implies that $\lim_{\epsilon \to 0} {\delta}/{\epsilon} =0$.

\medskip

\textbf{Case 2: } $\Delta < \delta^{1/3}n$.
Using a greedy strategy one can choose a matching $\mathcal{M}$ with at least $\left( \delta n^2/2 \right)/ \left( 2 \delta^{1/3} n \right)=\delta^{2/3}n/{4}$ edges from $B_1$.
Let $u_1v_1,...,u_mv_m$ be the edges in $\mathcal{M}$.
Since $G$ is $K_3$-free. Therefore, we have $d_2(u_i) + d_2(v_i)\le |V_2|$ and hence
\[
|M|\ge \sum_{i=1}^{m}\left(2|V_2|- d_2(u_i) - d_2(v_i) \right)\ge m |V_2|\ge \frac{\delta^{2/3}}{4}n \times \frac{n}{3}=  \frac{\delta^{2/3}}{12}n^2
\]
Similarly we obtain that $\lim_{\epsilon \to 0} {\delta}/{\epsilon} =0$.

Our example above shows that for $K_3$-free graphs there is no absolute constant $C>0$ such that $\delta/ \epsilon\ge C$ holds for all $\epsilon >0$.


\section{Acknowledgement}
The author is very grateful to Dhruv Mubayi for his suggestions that have greatly improved the presentation.

\bibliographystyle{unsrt}
\bibliography{cancel_stab}
\end{document}